\documentclass[10pt,a4paper]{amsart}

\usepackage[utf8]{inputenc}

\usepackage{amsfonts}

\usepackage{latexsym}

\usepackage{amssymb}

\usepackage{amscd}

\usepackage{hyperref}

\usepackage{amsmath}

\input xy
\xyoption{all}

\theoremstyle{plain}
\newtheorem{theorem*}{Theorem}
\newtheorem{theorem}{Theorem}[section]

\newtheorem{proposition}[theorem]{Proposition}
\newtheorem{lemma}[theorem]{Lemma}
\newtheorem{corollary}[theorem]{Corollary}

\theoremstyle{definition}
\newtheorem{definition}[theorem]{Definition}
\theoremstyle{remark}
\newtheorem{remark}[theorem]{Remark}

\numberwithin{equation}{section}

\def\F{\mathbf{F}}

\def\P{\mathbf{P}}
\def\Z{\mathbf{Z}}

\def\Q{\mathbf{Q}}
\def\N{\mathbf{N}}
\def\A{\mathbf{A}}
\def\d{\mathbf{d}}

\def\Q{\mathbf{Q}}

\def\1{\mathbf{1}}

\def\a{\mathbf{a}}
\def\x{\mathbf{x}}

\def\b{\mathbf{b}}
\def\f{\mathbf{f}}

\def\HH{\mathsf{H}}
\def\AA{\mathsf{A}}

\def \L {\mathcal{L}}
\def \H {\mathcal{H}}

\def \p {\mathfrak{p}}

\DeclareMathOperator{\No}{N}

\DeclareMathOperator{\Norm}{N}

\DeclareMathOperator{\Spec}{Spec}

\def \bgamma {\boldsymbol\gamma}

\author{R\'egis Blache}
\address{\'Equipe LAMIA,
INSP\'E de la Guadeloupe\\
Morne Ferret
97139 Les Abymes F.W.I.}
\email{regis.blache@univ-antilles.fr}

\title[Generic ordinarity]{Generic ordinarity for abelian coverings of the projective line}

\begin{document}

\begin{abstract}
We show that abelian coverings of the projective line of order prime to $p$ are generically $\mu$-ordinary in characteristic $p$. 

The images of the irreducible components of Hurwitz spaces of abelian coverings of the projective line by the Torelli morphism lie in some Shimura varieties. The stratification by Newton polygons of these varieties is known, and we show that the generic Newton polygon for the Hurwitz space coincides with the generic (or $\mu$-ordinary) Newton polygon of the smallest Shimura variety that contains its image.

In order to do this, we compute the generic Newton polygons for $L$-functions associated to multiplicative character sums over the projective line.
\end{abstract}


\maketitle

\section{Introduction}

Let $p$ denote a prime. The last decades have seen many results about the intersection of the open Torelli locus in characteristic $p$ with the stratifications of the moduli space $\AA_g$ of principally polarized abelian varieties of dimension $g$ by some discrete invariants such as its $p$-rank, $p$-torsion subgroup up to isomorphism (the Ekedahl-Oort strata) or $p$-divisible group up to isogeny (the Newton strata).

In this paper, we focus on this last stratification. It is well-known that the generic Newton polygon for smooth genus $g$ curves is the ordinary one, consisting of two segments of length $g$ and respective slopes $0$ and $1$. In other words, the open Torelli locus in $\AA_g$ intersects the open Newton stratum, curves are generically ordinary.

We restrict our attention to abelian covers of the projective space. We fix an abelian group $G$ of order prime to $p$; the irreducible components of the Hurwitz space $\HH_G$ of $G$-covers of the projective line correspond to monodromy data $\bgamma=(G,N,\a)$, where $N$ is the number of branch points in the covering, and $\a=(a_1,\ldots,a_N)$ is the inertia type consisting of non zero elements of $G$ that generate it and whose sum is trivial. We denote by $\HH(\bgamma)$ the irreducible component associated to $\bgamma$.

For any curve in the Hurwitz space $\HH(\bgamma)$, the action of $G$ extends to its Jacobian, and its image by the Torelli morphism lies in some Shimura variety of PEL-type. This action restricts the admissible Newton polygons \cite{rari}, and the non-empty Newton strata of such varieties are known \cite{viwe}.

Our main result asserts that the intersection of the image of $\HH(\bgamma)$ with the open Newton stratum (corresponding to the lowest, or $\mu$-ordinary polygon) of the smallest Shimura variety containing it is always non-empty. In other words

\begin{theorem}
\label{main}
Curves in the Hurwitz space $\HH(\bgamma)$ are generically $\mu$-ordinary.
\end{theorem}

This is a generalization of known results. If instead of considering the Newton strata, we consider the --coarser-- $p$-rank strata, we get a theorem of Bouw \cite[Theorem 6.1]{bouw}. Moreover, this result is already known when there are at most $5$ branch points, and the characteristic is large enough \cite[Theorem 1.1]{limasi}. 

It is also known that for $20$ exceptional monodromy data, the image of $\HH(\bgamma)$ is open and dense in the corresponding Shimura variety \cite{moss}. In this case it is shown in \cite{lmpt,lmpt2} that all Newton polygons that occur in the Shimura variety are the polygons of smooth covers of the projective line.

Note also that from \cite[Theorem 1.3.7]{most} the open strata for the Ekedahl-Oort and Newton stratifications agree in a Shimura variety; thus another consequence of the above result, from the point of view of $p$-torsion groups, is that curves in the Hurwitz space $\HH(\bgamma)$ are generically $[p]$-ordinary. 

Finally, another interesting application of this result is that it extends the variety of Newton polygons coming from a Jacobian, and gives many new examples of unlikely intersections between the Torelli locus and some Newton strata \cite{lmpt2}.

We end this introduction with a description of Theorem \ref{main2}, which is the other main result of this article, and we reduce the proof of Theorem \ref{main} to its proof.

We fix some integers $d\geq 2$ coprime to $p$, and $\a=(a_1,\ldots,a_N)$ such that $d$ divides $\sum a_i$ and $\gcd(d,a_1,\ldots,a_N)=1$. We set $\gamma:=(d,N,\a)$. 

Let $\chi$ denote a character of order $d$ of $\overline{\F}_p^\times$. When $g$ runs over the family of polynomials of the form $\prod_{i=1}^N (x-\alpha_i)^{a_i}$ with pairwise distinct $\alpha_i$, the Newton polygons of the multiplicative $L$-functions $L(\chi,g;T)$ (see Section \ref{sec1} for a precise definition) attain a lower bound over some Zariski open subset; this is Grothendieck's specialization theorem \cite[Theorem 2.3.1]{kasl}. The lower bound is the generic Newton polygon associated to this family. 

In theorem \ref{main2}, we determine this polygon for all $d$, $p$ and $\a$ as above, and the corresponding Hasse polynomial $\H_{\gamma,p}$, the equation in $\F_p[\alpha_1,\ldots,\alpha_N]$ of the hypersurface whose complementary is the open Newton stratum. 

This polygon (see Definition \ref{defpoly}) is a part of the $\mu$-ordinary polygon, corresponding to an orbit of multiplication by $p$ in the group $\Z/d\Z$. It is already defined (in a different way) in \cite[Section 5]{doll}.

Finally, let us reduce Theorem \ref{main} to Theorem \ref{main2}. First note that we can reduce to the cyclic case from \cite[Section 3]{limasi}. If we consider a cyclic monodromy datum $\bgamma=(\Z/d\Z,N,\a)$, a curve $C$ in the corresponding Hurwitz space has an equation of the form $y^d=\prod_{i=1}^N (x-\alpha_i)^{a_i}$ with pairwise distinct $\alpha_i$ from \cite[Lemma 7.1]{prie}. 

Now observe that the $L$-function $L(C;T)$ of the curve $C$ (the numerator of its zeta function) factors as the product of the $L$-functions associated to the character $\chi$ and the functions $g^k$, $1\leq k\leq d-1$. As a consequence, the $q$-adic Newton polygon of $L(C,T)$ (which is nothing but the Newton polygon of the Jacobian of $C$) is the concatenation of the Newton polygons of the multiplicative $L$-functions $L(\chi,g^k;T)$. Since the concatenation of the generic Newton polygons for these multiplicative $L$-functions is the $\mu$-ordinary polynomial for $\bgamma$, we get the result.

Note that some equivalents of Theorem \ref{main2} are known for $L$-functions associated to additive character sums, coming from a $p$-cyclic covering of the projective line. The generic Newton polygon is determined in \cite{bf} when the covering is ramified at one point, and the characteristic is large compared to the degree; for a general $p$-cover it is shown to admit a limit when $p$ tends to infinity in \cite{lizh}. But the generic Newton polygon remains mysterious in the general case: it is much more dependant on the characteristic than in the multiplicative case, where it depends only on the residue of $p$ modulo $d$ and $\bgamma$.

Let us say a word about our techniques: they are elementary, and rather old-fashioned. We first use a method dating back to Davenport and Hasse \cite{daha} in order to express the coefficients of the $L$-functions in terms of character sums over higher dimensional projective space. Once this is done, we write these last sums in terms of Gauss sums, exploiting ideas of Ax \cite{ax}, and we deduce the lower polygon from the valuations of these sums. This is done in Section \ref{sec1}. In order to study the terms in the expression that participate to the principal parts of the coefficients of the $L$-functions, we study the elements with minimal $p$-weight among the solutions of some congruences in Section \ref{sec2}. The last section is devoted to write these principal parts of the coefficients of the $L$-function, and to show that they are generically non zero with the help of the above results and Stickelberger's congruence \cite{stic}.

\section{A lower bound for Newton polygons of $L$-functions associated to multiplicative character sums}
\label{sec1}

In this section, we fix two integers $d\geq 2$, $N\geq 3$, and a vector $\a:=(a_1,\ldots,a_N)\in\{1,\ldots,d-1\}^N$ such that $d$ divides $A$, the sum of its coordinates. We set $\gamma:=(d,N,\a)$.

We choose a prime $p$ that does not divide $d$. Let $r$ denote the multiplicative order of $p$ in $(\Z/d\Z)^\times$.

Let $\F_q$ denote a finite field with $q=p^m\equiv 1\mod d$ and $q\geq N$; in particular $r$ divides $m$. Let $g\in \F_q[x]$ denote a polynomial of the form
\[
g(x)=\prod_{i=1}^N (x-\alpha_i)^{a_i},
\]
where $\alpha_1,\ldots,\alpha_N\in\F_q$ are pairwise distinct.

The field $\Q_p(\zeta_{q-1})$ is the unramified extension of degree $m$ of the field $\Q_p$ of $p$-adic numbers. The residue field of its ring of integers $\Z_p[\zeta_{q-1}]$ is $\F_q$, and we denote by $\omega$ the map on $\F_q^\times$ which is the section of reduction modulo $p$ whose image is the Teichmüller subgroup, ie the subgroup of $(q-1)$th roots of unity. We extend $\omega$ by $\omega(0)=0$. This is a generator of the group of multiplicative characters on $\F_q^\times$ taking their values in $\Q_p(\zeta_{q-1})$, and $\chi:=\omega^{\frac{q-1}{d}}$ is a generator of the group of multiplicative characters of order $d$, that we extend to $\F_q$ by setting $\chi(0):=0$. For any $k\geq 1$, we extend the character $\chi$ to the degree $k$ extension of $\F_q$, with the help of the norm: we set $\chi_k:=\chi\circ\No_{\F_{q^k}/\F_q}$ on $\F_{q^k}$. 

Let $G(x_0,x_1):=x_1^A g(x_0/x_1)$ denote the polynomial obtained by homogeneization of $g$. Since the degree of $G$ is a multiple of $d$, the value $\chi_k(G(x_0,x_1))$ does not depend on the choice of the representative $(x_0,x_1)$ of the point $(x_0:x_1)\in\P^1(\F_{q^k})$. Thus we can associate to $g$ and $\chi$ a family of character sums $(s_k)_{k\geq 1}$ over the projective line defined by (note that for the point at infinity, we have $G(1,0)=1$)
\[
s_k:=\sum_{(x_0:x_1)\in\P^1(\F_{q^k})}\chi_k(G(x_0,x_1))=1+\sum_{x\in\F_{q^k}}\chi_k(g(x))
\]
We define a $L$-fonction from this family
\[
L(g,\chi;T):= \exp\left(\sum_{k\geq 1} s_k\frac{T^k}{k}\right)
\]
The aim of this section is to give a lower bound for the Newton polygon (for the $q$-adic valuation) of this $L$-function. 

In order to do this, we first give an expression for its coefficients, in terms of character sums over higher dimensional projective spaces. Then we give a $p$-adic expression for these last sums, in order to lower bound their valuations. 

\subsection{A new expression for the coefficients of the $L$-function} 
 
We shall use the Hasse-Davenport method to rewrite the coefficients of the $L$-function. This section is largely inspired from \cite[Section 2]{kag2}.

First recall the expression of the $L$-function as an Eulerian product
\[
L(g,\chi;T)=\prod_{\p} \frac{1}{1-\chi_{\deg \p}(g(\p))T^{\deg(\p)}}
\]
where $\p$ runs over the places of $\F_q(x)$, ie over the monic irreducible polynomials in $\F_q[x]$ and the place $(1/x)$ corresponding to the place at infinity. Note that if $\pi$ is such a polynomial, corresponding to $\p$, we have set $\chi_{\deg \p}(g(\p))=\chi_{\deg \p}(g(\theta))$ for any root $\theta$ of $\pi$ in $\F_{q^{\deg \pi}}$. Note also that for $\p=(1/x)$ we have $\chi(g(\p))=1$ as we have already seen.

From this expression, we deduce that for any $n\geq 1$, the degree $n$ coefficient of the $L$-function can be written
\[
L_n=\sum_{(\p_1,\ldots,\p_k)}\prod_{i=1}^k \chi_{\deg \p_i}(g(\p_i))
\]
where the sum runs over all lists of places $(\p_1,\ldots,\p_k)$ with $\sum \deg\p_i=n$. Since we have $\chi(g(\p))=1$ for the place at infinity, its presence in the list does not change anything to the product, and we can replace the summation set by the set of lists of monic irreducible polynomials $(\pi_1,\ldots,\pi_k)$ with $\sum \deg\pi_i\leq n$. Thus we can rewrite
\[
L_n=\sum_{(\pi_1,\ldots,\pi_k)} \chi\left(\prod_{i=1}^k \No_{\F_{q^{\deg \pi_i}}/\F_q}(g(\theta_i))\right)=\sum_{(\pi_1,\ldots,\pi_k)} \chi\left(\prod_\theta g(\theta)\right)
\]
where the last product is over the roots of the polynomials $\pi_1,\ldots,\pi_k$ (note that if $\theta_i$ is a root of $\pi_i$, then its roots are the $\theta_i^{q^j}$ for $0\leq j\leq\deg\pi_i-1$ and we can replace the norm by the product over all roots).

From the unique factorization theorem, we deduce that the above sum can be indexed by the unitary polynomials of degree at most $n$. The polynomial (in the variable $t$) $x_0+x_1t+\ldots+x_nt^n$ corresponds to the point $(x_0:\ldots:x_n)$ of the projective space $\P^n(\F_q)$, where we have chosen the representatives whose last nonzero coordinate is $1$. Moreover, the product over the roots can be rewritten in terms of the coefficients of the polynomial
\[
\prod_\theta g(\theta)=\prod_{i=1}^N \prod_\theta (\theta-\alpha_i)^{a_i}=(-1)^{A}\prod_{i=1}^n h(\alpha_i)^{a_i}
\]
Summing up, we get the expression (note that since $d$ divides $\sum a_i$, the power of $-1$ disappears and the product does not depend on the chosen representative for a point of projective space)
\[
L_n=\sum_{x=(x_0:\ldots:x_n)\in\P^n(\F_q)} \prod_{i=1}^N \chi^{a_i}(\ell_i(x)),~ \ell_i(x):=x_0+\alpha_ix_1+\ldots+\alpha_i^nx_n
\]

\subsection{An expression of the coefficients in terms of Gauss sums}

In order to evaluate the divisibility of the coefficients, we apply some techniques that have already been used in such situations. First we transform them into mixed sums as in \cite{doll}. Then we use a slight modification of the method of Ax \cite{ax}.

It will be easier to evaluate a sum on affine space. Since we have set $\chi(0)=0$, we have the simple relation $(q-1)L_n=S_n$, where
\[
S_n=\sum_{x\in \F_q^{n+1}} \prod_{i=1}^N \chi^{a_i}(\ell_i(x))
\]

Recall that $\alpha_1,\ldots,\alpha_N$ are pairwise distinct. 

First assume that we have $n+1\geq N$; the family $\{\ell_1,\ldots,\ell_{N}\}$ of linear forms is free in the dual of $\F_q^{n+1}$; in this case the sum $S_n$ is zero from the orthogonality relation for multiplicative characters. We deduce the following well-known fact

\begin{lemma}
The $L$-function $L(g,\chi;T)$ is a polynomial of degree at most $N-2$.
\end{lemma}

From now on, we assume that we have $n+1<N$. Then the family $\{\ell_1,\ldots,\ell_{n+1}\}$ is a basis for the dual of $\F_q^{n+1}$; we easily verify that for all $n+2\leq i\leq N$ we have 
\[
\ell_i=\sum_{j=1}^{n+1} \alpha_{ij}\ell_j,
\] 
where $\alpha_{ij}=\L_j(\alpha_i)$ is the value ar $\alpha_i$ of the $j$-th Lagrange interpolation polynomial associated to $\alpha_1,\ldots,\alpha_{n+1}$.

We can write the sum $S_n$ in the dual basis of $\{\ell_1,\ldots,\ell_{n+1}\}$; using the fact that the value at $0$ of a non trivial multiplicative character is zero, we get
\[
S_n=\sum_{y\in (\F_q^\times)^{n+1}} \prod_{j=1}^{n+1} \chi^{a_j}(y_j)\prod_{i=n+2}^N \chi^{a_i}(m_i(y)),~ m_i(y):=\sum_{j=1}^{n+1} \alpha_{ij} y_j
\]

We now transform the sum $S_n$ into a mixed character sum. We fix once and for all an element $\pi$ in $\Q_p(\zeta_p)$ such that $\pi^{p-1}=-p$. We know that there exists an unique $p$-root of unity with $\zeta_p\equiv 1+\pi\mod \pi^2$. It defines an additive character $\psi_0$ of $\F_p$ by $\psi_0(1)=\zeta_p$, and, composed with the absolute trace, an additive character $\psi$ of $\F_q$.

\begin{definition}
\label{gausssums}
For any integer $i\neq 0$, we define the Gauss sum on $\F_q$, associated to the multiplicative character $\omega^i$, as
\[
g(i):=\sum_{x\in \F_q}\psi(x)\omega^i(x)
\]
and we set as a convention $g(0):=-1$.
\end{definition}

For any $x\in\F_q$, and $i$ as above, we have the equality
\[
\omega^i(x)=g(-i)^{-1}\sum_{y\in \F_q^\times} \omega^{-i}(y)\psi(xy)
\]
as can be seen from the variable change $z=xy$ in the sum when $x\neq 0$, and from the orthogonality relation for multiplicative characters else.

Let $b_i:=\frac{q-1}{d}a_i$, $1\leq i\leq N$, so that $\chi^{a_i}=\omega^{b_i}$. If we replace the values $\chi^{a_i}(m_i(y))$ by the sums given in the preceding expression, and set $\b_n:=(b_1,\ldots,b_{n+1},q-1-b_{n+2},\ldots,q-1-b_N)$, $\omega^{\b_n}(y):=\prod_{j=1}^{n+1}\omega^{b_j}(y_j)\prod_{i=n+2}^{N}\omega^{-b_i}(y_i)$ for $y=(y_1,\ldots,y_N)$, then we can rewrite the sum
\begin{eqnarray*}
S_n & = & \left(\prod_{i=n+2}^N g(-b_i)\right)^{-1}\sum_{y\in(\F_q^\times)^{N}}\omega^{\b_n}(y)\prod_{i=n+2}^{N}\psi(y_im_i(y_1,\ldots,y_{n+1})) \\
& = & \left(\prod_{i=n+2}^N g(-b_i)\right)^{-1}\sum_{y\in(\F_q^\times)^{N}}\omega^{\b_n}(y)\prod_{i=n+2}^{N}\prod_{j=1}^{n+1}\psi(\alpha_{ij}y_iy_j) \\
\end{eqnarray*}

\begin{remark}
Note that the $b_i$ inherit some properties from those of the $a_i$; in particular none is divisible by $q-1$, but their sum is.
\end{remark}

We turn to our slight modification of Ax' method (see \cite[Section 6]{adsp}). We introduce a polynomial $C:=\sum_{k=0}^{q-2}c_kT^k\in \Q_p(\zeta_p,\zeta_{q-1})[T]$: it is the unique degree $q-2$ polynomial which takes the value $\psi(x)$ at $\omega(x)$ for any $x\in\F_q^\times$; in other words, it is the Lagrange interpolation polynomial for the additive character $\psi$ at the points of the Teichmüller subgroup of $\Q_p(\zeta_{q-1})$. We easily check that its coefficients can be written in terms of Gauss sums
\[
c_k=\frac{g(-k)}{q-1},~1\leq k\leq q-2,~c_0=\frac{1}{1-q}=\frac{g(0)}{q-1}
\] 
If we replace the values of the additive character, using the polynomial $C$, we get 
\[
S_n = \left(\prod_{i=n+2}^N g(-b_i)\right)^{-1}\sum_{y\in(\F_q^\times)^{N}}\omega^{\b_n}(y)\prod_{i=n+2}^{N}\prod_{j=1}^{n+1}\sum_{k=0}^{q-2}c_k\omega^k(\alpha_{ij}y_jy_i)
\]
We introduce the set $E_n(q):=\{0,\ldots,q-2\}^{(N-n-1)(n+1)}$, and write its elements in the form $U=(u_{ij})$, $1\leq j\leq n+1$, $n+2\leq i\leq N$. For any such $U$, we set
\[
\f_n(U):=\left(\sum_{i=n+2}^N u_{i1},\ldots,\sum_{i=n+2}^N u_{in+1},\sum_{j=1}^{n+1} u_{n+2j},\ldots,\sum_{j=1}^{n+1} u_{Nj}\right)
\]
Interchanging the internal sum with the products in the expression above, we get
\[
S_n = \left(\prod_{i=n+2}^N g(-b_i)\right)^{-1}\sum_{y\in(\F_q^\times)^{N}}\sum_{U\in E_n(q)}\omega^{\f_n(U)+\b_n}(y) \prod_{j=1}^{n+1}\prod_{i=n+2}^{N} c_{u_{ij}}\omega^{u_{ij}}(\alpha_{ij})
\]
Finally, if we exchange the two sums, and use the orthogonality relations on multiplicative characters, we see that since $\omega$ has order $q-1$, the term corresponding to $U$ above is zero, except when $q-1$ divides all the coordinates of the vector $\f_n(U)+\b_n$. We are led to introduce the set
\[
A_n(\gamma,q) := \left\{U\in E_n(q),~ \f_n(U)+\b_n\equiv 0\mod q-1\right\}
\]
and we get the expression
\begin{equation}
\label{exprSn}
S_n = (q-1)^N\left(\prod_{i=n+2}^N g(-b_i)\right)^{-1}\sum_{U\in A_n(\gamma,q)}\prod_{j=1}^{n+1}\prod_{i=n+2}^{N} c_{u_{ij}}\omega^{u_{ij}}(\alpha_{ij}) 
\end{equation}

Note that the set $A_n(\gamma,q)$ contains $(q-1)^{n(N-n-2)}$ elements. This is easily verified by expressing each $u_{Nj}$ in terms of the $u_{n+2j},\ldots,u_{N-1j}$ for $1\leq j\leq n+1$, and each $u_{in+1}$ in terms of the $u_{i1},\ldots,u_{in}$ for $n+2\leq i\leq N$. We get two expressions for the term $u_{Nn+1}$, which agree since the sum $\sum_{i=1}^N b_i$ is a multiple of $q-1$. Thus there exists exactly one solution for any given choice of $(u_{ij})$, $1\leq j\leq n$, $n+2\leq i\leq N-1$.

Note in particular that for $n=N-2$, there is a unique solution, namely \[
(u_{1N},\ldots,u_{N-1N})=(q-1-b_1,\ldots,q-1-b_{N-1})
\]
In other words, we have expressed the degree $N-2$ coefficient of the $L$-function as a product of Gauss sums (and other non zero terms). We get the (well-known)

\begin{lemma}
The $L$-function $L(g,\chi;T)$ is a degree $N-2$ polynomial.
\end{lemma}

\subsection{A lower bound for the Newton polygon}

Here we use the expression (\ref{exprSn}) and a well-known result on the valuation of Gauss sums, to give a lower bound for the valuations of the coefficients of the $L$-function.

Recall that we have chosen a generator of the maximal ideal of the ring $\Z_p[\zeta_p,\zeta_{q-1}]$ as $\pi=\psi(1)-1$. In this case, for any integer $1\leq i\leq q-2$, we now from a theorem of Stickelberger that the $\pi$-adic valuation of the Gauss sum $g(-i)$ is $s_p(i)$, the sum of the digits of the base $p$ expansion of $i$.

We deduce that for all $0\leq u\leq q-2$, the degree $u$ coefficient of the polynomial $C$ has $\pi$-adic valuation $s_p(i)$, and the minoration
\[
v_{\pi}(S_n)\geq \min\left\{s_p(U)-\sum_{i=n+2}^N s_p(b_i),~ U\in A_n(\gamma,q)\right\},~ s_p(U):=\sum_{i,j} s_p(u_{ij})
\]
We are ready to give a lower bound for the valuation of $S_n$. In order to do this, we need a definition

\begin{definition}
The \emph{$p$-signature} of $\gamma=(d,N,\a)$ is the function \[\sigma:\N\rightarrow \{0,\ldots,N-2\},~
\sigma(t):=-1+\sum_{k=1}^N \left\langle\frac{p^ta_k}{d}\right\rangle
\]
where $\langle x\rangle$ is the fractional part of the rational number $x$.
\end{definition}

Note that it is a periodic function, with period $r$. With this at hand, we have

\begin{proposition}
\label{minoval}
For any $1\leq n\leq N-2$, the valuation of the degree $n$ coefficient of the $L$-function $L(g,\chi;T)$ satisfies
\[
v_{\pi}(L_n)\geq (p-1)\sum_{t=0}^{m-1}\max\{0,n-\sigma(t)\}=\frac{(p-1)m}{r}\sum_{t=0}^{m-1}\max\{0,n-\sigma(t)\}
\]
and this is an equality for $n=N-2$.
\end{proposition}

\begin{proof}
It is equivalent to work with the sum $S_n$, which has the same valuation.

From the expression above, it is sufficient to give a lower bound on the sums $s_p(U)$ for any $U\in A_n(\gamma,q)$.

Recall the following expression for the sum of base $p$ digits: for any $0\leq k\leq q-2$, we have
\[
s_p(k)=(p-1)\sum_{t=0}^{m-1} \left\langle\frac{p^t k}{q-1}\right\rangle
\]
We deduce the following
\[
s_p(U)=(p-1)\sum_{t=0}^{m-1} \sum_{i=n+2}^N \sum_{j=1}^{n+1} \left\langle\frac{p^t u_{ij}}{q-1}\right\rangle
\]
First fix some $1\leq j\leq n+1$; from  the definition of $A_n(\gamma,q)$, we know that $\sum_{i=n+2}^N p^tu_{ij}+p^tb_j$ is a positive multiple of $q-1$ for any $t$, and since we have $0<b_j<q-1$, we deduce the lower bound
\begin{equation}
\label{ineq1}
   \sum_{i=n+2}^N \left\langle\frac{p^t u_{ij}}{q-1}\right\rangle\geq 1-\left\langle\frac{p^t b_j}{q-1}\right\rangle 
\end{equation}

We get a first inequality (note that $\frac{a_i}{d}=\frac{b_i}{q-1}$ in order to use the $p$-signature $\sigma$)
\[
\sum_{j=1}^{n+1} \sum_{i=n+2}^N  \left\langle\frac{p^t u_{ij}}{q-1}\right\rangle\geq  n+1-\sum_{j=1}^{n+1} \left\langle\frac{p^t b_j}{q-1}\right\rangle=n-\sigma(t)+\sum_{i=n+2}^N \left\langle\frac{p^t b_i}{q-1}\right\rangle
\]
Now fix some $n+2\leq i\leq N$; then $\sum_{j=1}^{n+1} p^tu_{ij}+p^t(q-1-b_i)$ is a positive multiple of $q-1$ for any $t$, and we have (note that for any rational which is not an integer, we have $\langle -x\rangle=1-\langle x\rangle$)
\begin{equation}
\label{ineq2}
\sum_{j=1}^{n+1} \left\langle\frac{p^t u_{ij}}{q-1}\right\rangle+1-\left\langle\frac{p^t b_i}{q-1}\right\rangle\geq 1
\end{equation}
We get a second inequality
\[
\sum_{i=n+2}^N \sum_{j=1}^{n+1}   \left\langle\frac{p^t u_{ij}}{q-1}\right\rangle\geq  \sum_{i=n+2}^N \left\langle\frac{p^t b_i}{q-1}\right\rangle
\]
and, summing all these inequalities when $t$ varies, the minoration valid for any $U\in A_n(\gamma,q)$
\begin{eqnarray*}
s_p(U) & \geq & (p-1)\sum_{t=0}^{m-1} \left(\max\{0,n-\sigma(t)\}+\sum_{i=n+2}^N \left\langle\frac{p^t b_i}{q-1}\right\rangle\right)\\
& \geq & \sum_{i=n+2}^N s_p(b_i)+(p-1)\sum_{t=0}^{m-1} \max\{0,n-\sigma(t)\}
\end{eqnarray*}

This proves the inequality in the formula asserted in the proposition; the equality after it follows immediately from the $r$-periodicity of $\sigma$.

It remains to prove that we have an equality when $n=N-2$. In this case, we have already described the unique element of $A_n(\gamma,q)$, that is $(u_{1N},\ldots,u_{N-1N})=(q-1-b_1,\ldots,q-1-b_{N-1})$. Thus the first inequality in the proof above is an equality, and since our expression for the coefficient $L_{N-2}$ contains an unique term, both must have the same valuation.
\end{proof}

We are ready to define a combinatorial polygon (see the beginning of \cite[Section 5]{doll} for an equivalent definition)

\begin{definition}
\label{defpoly}
Let $\gamma,p$ be as above. If $e$ is the residue of $p$ modulo $d$, we denote by $\Pi(\gamma,e)$ the \emph{$\mu$-ordinary polygon} associated to this data. It is the graph of the continuous piecewise affine function $B$ such that $f(0)=0$ and for $1\leq n\leq N-2$, its slope over the interval $[n-1,n]$ is 
\[
\lambda_n=\frac{1}{r}\#\{0\leq i\leq r-1,~ \sigma(i)\leq n-1\}
\]
We denote by $\Sigma(\gamma,e)=\{\pi_1,\ldots,\pi_k\}$ the set of its vertices except the endpoints, where we write $\pi_i=(n_i,B(n_i))$.
\end{definition}

\begin{remark}
Let $1\leq n\leq N-3$; the point of abscissa $n$ is a vertex of this polygon if and only if the slopes over the intervals $[n-1,n]$ and $[n,n+1]$ are different, if and only if the map $\sigma$ takes the value $n$. As a consequence, the cardinality of $\Sigma(\gamma,e)$ is the number of values taken by $\sigma$ in $\{0,\ldots,N-3\}$, and the number of slopes is one more.
\end{remark}

\begin{remark}
\label{polygonorbit}
If we set $\gamma'=(d,N,p^k\a)$ (where the coordinates of $p^k\a$ are reduced modulo $d$), the new $p$-signature $\sigma'$ satisfies $\sigma'(t)=\sigma(t+k)$. In particular, they take the same values, and the associated Newton polygons are the same.

The union of these polygons when $k$ varies is the polygon denoted by $\mu(\mathfrak{o})$ in \cite[Proposition 4.3]{lmpt} when $\mathfrak{o}$ is the orbit of $1$ by multiplication by $p$ in $\Z/d\Z$, a piece of the $\mu$-ordinary polygon attached to the monodromy datum $\gamma$.
\end{remark}

This polygon is our lower bound

\begin{corollary}
\label{lowerbound}
Let $g\in \F_q[x]$ denote a polynomial as above. The Newton polygon of the associated $L$-function $L(g,\chi;T)$ for the $q$-adic valuation lies above the polygon $\Pi(\gamma,e)$, and their endpoints coincide.
\end{corollary}

\begin{proof}
First note that the two valuations are linked by $v_\pi=m(p-1)v_q$. 

From the above proposition, we have the inequality $v_q(L_n)\geq \frac{1}{r}\sum_{t=0}^{r-1}\max\{0,n-\sigma(t)\}$ valid for any $1\leq n\leq N-2$, and it is an equality when $n=N-2$. For $n=0$, the constant coefficient of the $L$-function is $1$, and both polygons start at the origin.

In order to conclude, we just have to remark
\begin{eqnarray*}
\sum_{t=0}^{r-1}\max\{0,n-\sigma(t)\}
& = & \sum_{k=0}^{n-1}(n-k)\#\{0\leq i\leq r-1,~\sigma(i)=k\}\\
& = & \sum_{k=0}^{n-1}\#\{0\leq i\leq r-1,~\sigma(i)\leq k\} = r\sum_{k=1}^{n} \lambda_k\\ 
\end{eqnarray*}
\end{proof}

\begin{remark}
This result is already known: except for the coincidence of the end-points, it is the simplest case of \cite[Theorem 5.1]{doll}.

We have decided to give a new proof since it provides most of the tools that we shall use in order to prove the more precise result that we need: the lower bound $\Pi(\gamma,e)$ is generically tight.  
\end{remark}

\section{Minimal solutions}
\label{sec2}

This section deals with the sets $\overline{A}_n(\gamma,q)$ consisting of the solutions of the system of congruences (slightly modified from the preceding section), and with the subsets of minimal solutions. We shall mainly show two results, both assuming that the map $\sigma$ takes the value $n$: the set $M_n(\gamma,q)$ is non-empty, and it can be described from the set $M_n(\gamma,p^r)$.

We first recall the notations introduced in the preceding section. We fix an integer $d\geq 2$, a prime $p$ not dividing $d$, and a power $q=p^m\equiv 1 \mod d$; we denote by $r$ the multiplicative order of $p$ modulo $d$, so that $m$ is a multiple of $r$. Let $\a:=(a_1,\ldots,a_N)\in\{1,\ldots,d-1\}^N$ be such that $d$ divides the sum $\sum_k a_k$. For $1\leq k\leq N$, let $b_k=\frac{q-1}{d}a_k$.

For any $1\leq n\leq N-2$, we define a map on $\overline{E}_n(q):=\{0,\ldots,q-1\}^{(N-n-1)(n+1)}$
\[
\f_n(U):=\left(\sum_{i=n+2}^N u_{i1},\ldots,\sum_{i=n+2}^N u_{in+1},\sum_{j=1}^{n+1} u_{n+2j},\ldots,\sum_{j=1}^{n+1} u_{Nj}\right)
\]
We set $\b_n:=(b_1,\ldots,b_{n+1},q-1-b_{n+2},\ldots,q-1-b_N)$ and consider the set of solutions of the system of the following $N$ congruences
\[
\overline{A}_n(\gamma,q) := \left\{U\in \overline{E}_n(q),~ \f_n(U)+\b_n\equiv 0\mod q-1\right\}
\]

We will focus on those having minimal $p$-weight $s_p(U)=\sum_{i,j} s_p(u_{ij})$.

\begin{definition}
The \emph{minimal solutions} are the solutions $U\in \overline{A}_n(\gamma,q)$ such that 
\[
s_p(U)= \sum_{i=n+2}^N s_p(b_i)+(p-1)\sum_{t=0}^{m-1}\max\{0,n-\sigma(t)\}
\]
We denote by $M_n(\gamma,q)$ the set of minimal solutions.
\end{definition}

\subsection{Characterizations of minimal solutions}

We give a series of properties of minimal solutions, that we shall use to prove our main results.

\begin{lemma}
\label{noq-1}
We have the inclusion $M_n(\gamma,q)\subset \{0,\ldots,q-2\}^{(N-n-1)(n+1)}$
\end{lemma}

\begin{proof}
We see from the proof of Proposition \ref{minoval} that a minimal solution is one with minimal $p$-weight $s_p(U)$. If $U$ was such a solution, with some $u_{ij}=q-1$, we could replace it by $0$ and still have a solution, but with strictly lower weight, contradicting the minimality.
\end{proof}

We introduce some objects, that are already used in \cite{bla},  \cite{phig}.

\begin{definition}
The \emph{shift} $\delta$ from $\{0,\ldots,q-1\}$ to itself, sends any integer $u$ to the remainder of the euclidean division of $pu$ by $q-1$, and fiwes $q-1$. 

We extend it coordinatewise to any vector of elements of $\{0,\ldots,q-2\}$ 
\end{definition}

Note that $\delta$ permutes cyclically the digits of the base $p$ expansion, and that $\delta^m$ is the identity; as a consequence, it preserves the $p$-weight. It sends $\overline{A}_n(\gamma,q)$ to $\overline{A}_n(\delta\gamma,q)$, and $M_n(\gamma,q)$ to $M_n(\delta\gamma,q)$. Moreover, from its very definition, we have
\[
\forall t\in\{0,\ldots,m-1\},~ u\in\{0,\ldots,q-2\}, ~\left\langle\frac{p^t u}{q-1}\right\rangle
=\frac{\delta^t u}{q-1}
\]

Here is a first characterization for minimal solutions.

\begin{lemma}
\label{charmin}
A solution $U\in A_n(\gamma,q)$ is minimal if and only if we have for all $t\in\{0,\ldots,m-1\}$, the equalities
\[
\left\{
\begin{array}{llllll}
\forall j\in \{1,\ldots,n+1\},& \sum_{i=n+2}^{N}\delta^t u_{ij} & = & q-1-\delta^t b_j & {\rm if} & \sigma(t)\leq n \\
\forall i\in \{n+2,\ldots,N\},& \sum_{j=1}^{n+1}\delta^t u_{ij} & = & \delta^t b_i & {\rm if} & \sigma(t)\geq n  \\
\end{array}
\right.
\]
\end{lemma}

\begin{proof}
Recall from the proof of Proposition \ref{minoval} that a solution $U$ is minimal if, and only if, for all $0\leq t\leq m-1$, we have the equality
\[
\sum_{i=n+2}^{N}\sum_{j=1}^{n+1}\frac{\delta^t u_{ij}}{q-1}=\max\left\{n+1-\sum_{j=1}^{n+1}\frac{\delta^t b_j}{q-1},\sum_{i=n+2}^{N}\frac{\delta^t b_i}{q-1}\right\}
\]
First note that the above maximum is the first term when $\sigma(t)\leq n$, and the second when $\sigma(t)\geq n$.

If $\sigma(t)\leq n$, then the $n+1$ inequalities (\ref{ineq1}) must be equalities. If $\sigma(t)\geq n$, then the $N-n-1$ inequalities (\ref{ineq2}) must be equalities. Thus the equalities above are satisfied by a minimal solution.

The converse follows easily from the expression
\[
s_p(U)=\sum_{j=1}^{n+1}\sum_{i=n+2}^{N}s_p(u_{ij})=(p-1)\sum_{j=1}^{n+1}\sum_{i=n+2}^{N}\sum_{t=0}^{m-1}\frac{\delta^t u_{ij}}{q-1}
\]
\end{proof}

In order to give a second characterisation, we follow \cite{phig}, and associate a family of vectors with positive integer coordinates to any solution. 

\begin{definition}
Let $U\in A_n(\gamma,q)$; its \emph{support} is the map from $\Z/m\Z$ to $\N^N$ defined by
\[
\varphi_U(t):=\frac{1}{q-1}(\f_n(\delta^t U)+\delta^t\b_n)
\]
\end{definition}

We also need some sets that will describe the supports of minimal solutions

\begin{definition}
Let $0\leq t\leq m-1$; we define the set $M_t$ by
\begin{itemize}
    \item[(i)] if $\sigma(t)\leq n$, then
 \[
 M_t:=
\left\{(x_1,\ldots,x_N)\in \N^N,~ x_1=\ldots=x_{n+1}=1,~ \sum_{i=n+2}^N x_i=N-\sigma(t)-1\right\}
\]
    \item[(ii)] if $\sigma(t)\geq n$, then
 \[
 M_t:=
\left\{(x_1,\ldots,x_N)\in \N^N,~ \sum_{j=1}^{n+1} x_j=\sigma(t)+1,~ x_{n+2}=\ldots=x_{N}=1\right\}
\]
\end{itemize}
\end{definition}

\begin{remark}
\label{rem1}
Note that $M_t=\{(1,\ldots,1)\}$ when $\sigma(t)=n$, and that for any element in $M_t$, we have
\[
\sum_{j=1}^{n+1} x_j-\sum_{i=n+2}^N x_i=\sigma(t)+2-N+n
\]
\end{remark}

Then we have

\begin{lemma}
\label{suppmin}
A solution $U\in A_n(\gamma,q)$ is minimal if, and only if we have $\varphi_U(t)\in M_t$ for all $0\leq t\leq m-1$.
\end{lemma}

\begin{proof}
Assume that we have $\varphi_U(t)\in M_t$ for all $0\leq t\leq m-1$. From the definition of the support, all the equalities from the preceding lemma are satisfied, and the solution is minimal.

When $U$ is a minimal solution, we compute the sum of the coordinates of $\varphi_U(t)$.

Assume first that we have $\sigma(t)\leq n$; again from the preceding lemma, the first $n+1$ coordinates of $\varphi_U(t)$ equal $1$, and the sum of the last $N-n-1$ is
\begin{eqnarray*}
& & N-n-1+\frac{1}{q-1}(\sum_{i=n+2}^N\sum_{j=1}^{n+1} \delta^t u_{ij}-\sum_{i=n+2}^N\delta^t b_i)\\
& = & N-n-1+\frac{1}{q-1}(\sum_{j=1}^{n+1} (q-1-\delta^t b_j)-\sum_{i=n+2}^N\delta^t b_i) =  N - \sigma(t) -1\\
\end{eqnarray*}
and the result is true in this case.

If we have $\sigma(t)\geq n$, the last $N-n-1$ coordinates of $\varphi_U(t)$ equal $1$, and the sum of the first $n+1$ is
\[
\frac{1}{q-1}(\sum_{i=n+2}^N\sum_{j=1}^{n+1} \delta^t u_{ij}+\sum_{j=1}^{n+1}\delta^t b_j) =  \frac{1}{q-1}(\sum_{i=n+2}^N \delta^t b_i+\sum_{j=1}^{n+1}\delta^t b_j) =  \sigma(t) +1
\]
and this is the announced result.
\end{proof}

\subsection{Existence of minimal solutions}

We have worked properties of minimal solutions so far, bur we still do not know if such solutions exist. This is the main result of this section.

We consider the base $p$ expansions of minimal solutions, and show that their digits must be solutions of certain systems of equations. Then it is sufficient to show that there exist such systems having solutions in $\{0,\ldots,p-1\}$.

For any $u\in\{0,\ldots,q-1\}$, its base $p$ expansion is
\[
u = \sum_{k=0}^{m-1} p^k u^{(k)},~ 0\leq u^{(k)} \leq p-1
\]
Note that we shall usually consider the superscript $(k)$ as an element of $\Z/m\Z$ in the following.

Let us begin with easy properties of the maps defined above.

\begin{lemma}
\label{shiftprop}
We have the following equalities
\begin{itemize}
\item[(i)] for any $0\leq u\leq q-1$, $p\delta^t u-\delta^{t+1}u=(q-1)u^{(m-t-1)}$;
\item[(ii)] for any $0\leq t\leq m-1$, $p(\sigma(t)+1)-(\sigma(t+1)+1)=\sum_{k=1}^N b_k^{(m-t-1)}$
\end{itemize}
\end{lemma}

\begin{proof}
The first assertion comes from the equalities $\delta^t u=\sum_{k=0}^{m-1} p^k u^{(k-t)}$.

For the second, remark that we have
\[
\sigma(t)+1=\sum_{k=1}^N \left\langle \frac{p^t b_k}{q-1} \right\rangle =\sum_{k=1}^N \frac{\delta^t b_k}{q-1}
\]
Thus the second assertion comes from the first one.
\end{proof}

We first give a characterization of the base $p$ digits of minimal solutions

\begin{lemma}
\label{red1}
There exists a minimal solution in $A_n(\gamma,q)$ if and only if there exists $(\x(0),\ldots,\x(m-1))\in M_0\times\ldots\times M_{m-1}$ -- where we set $\x(i):=(x_1(i),\ldots,x_N(i))$-- such that for all $0\leq t\leq m-1$, the system $({\rm E}_t)$ of $N$ equations
\[
\left\{
\begin{array}{llll}
\sum_{i=n+2}^N u_{ij}^{(m-t-1)} +b_{j}^{(m-t-1)} & = & px_j(t)-x_j(t+1) & 1\leq j\leq n+1 \\ 
\sum_{j=1}^{n+1} u_{ij}^{(m-t-1)}+(p-1-b_{i}^{(m-t-1)}) & = & px_i(t)-x_i(t+1) & n+2\leq i\leq N \\ 
\end{array}
\right.
\]
admits a solution $(u_{ij}^{(m-t-1)})_{i,j}\in \{0,\ldots,p-1\}^{(N-n-1)(n+1)}$.
\end{lemma}

\begin{proof}
Assume that there exists a minimal solution $U\in A_n(\gamma,q)$, and set $\x(t)=\varphi_U(t)$ for all $0\leq t\leq m-1$; from lemma \ref{suppmin}, we have $\x(t)\in M_t$ for all $t$. The right hand sides of the equations are the $N$ coordinates of $p\x(t)-\x(t+1)$. But we have $p\x(t)-\x(t+1)=p\varphi_U(t) -\varphi_U(t+1)$, and we can use lemma \ref{shiftprop} (i) to give another expression for the coordinates of this vector. We get $\f_n(U^{(m-t-1)})+\b_n^{(m-t-1)}$, where the digits are taken coordinate by coordinate; these are the hand sides of the equations (note that the base $p$ digits of $q-1-u$ are the $p-1-u^{(k)}$).

Conversely, fix such an $(\x(0),\ldots,\x(m-1))$, and let $(u_{ij}^{(m-t-1)})_{i,j}\in \{0,\ldots,p-1\}^{(N-n-1)(n+1)}$ denote a solution of $(E_t)$ in $\{0,\ldots,p-1\}^{(N-n-1)(n+1)}$. 

For any $n+2\leq i\leq N$, $1\leq j\leq n+1$, we define $u_{ij}:=\sum_{t=0}^{m-1} p^tu_{ij}^{(t)}$ and $U=(u_{ij})$. We have $\f(U)+\b_n=(q-1)\x(0)$ since the equations in the $(p^{m-t-1}E_t)$ telescope, and $\varphi_U(0)=\x(0)$. In the same way, we get $\varphi_U(t)=\x(t)$ for all $t$ and $U$ is minimal from lemma \ref{suppmin}.
\end{proof}

Here is a first reduction in the search of a solution

\begin{lemma}
\label{red2}
If there exists $\x=(\x(0),\ldots,\x(m-1))\in M_0\times\ldots\times M_{m-1}$, which satisfy the systems
\[
({\rm I}_t)\left\{
\begin{array}{llll}
x_j(t+1) & \leq & px_j(t) - b_{j}^{(m-t-1)} & 1\leq j\leq n+1 \\ 
x_i(t+1)  & \leq & px_i(t)-(p-1-b_{i}^{(m-t-1)}) & n+2\leq i\leq N \\ 
\end{array}
\right.
\]
then there exists a minimal solution.
\end{lemma}

\begin{proof}
Assume that there exists such $\x$. We get $mN$ non negative integers 
\[
y_j(t):=px_j(t) - b_{j}^{(m-t-1)}-x_j(t+1),~ 1\leq j\leq n+1,
\]
\[
y_i(t):=px_i(t)-(p-1-b_{i}^{(m-t-1)})-x_i(t+1),~ n+2\leq i \leq N
\]
We fix some $t$. Note that the system $({\rm E}_t)$ can be rewritten so that the above numbers are the right hand sides of the equations, and from Lemma \ref{red1}, it is sufficient to show that this system admits a solution $(c_{ij})\in\{0,\ldots,p-1\}^{(N-n-1)(n+1)}$.

First note that we have $M=\sum_{j=1}^{n+1} y_j(t)=\sum_{i=n+2}^N y_i(t)$ for all $t$; this comes from Remark \ref{rem1} and the second assertion of Lemma \ref{shiftprop}.

We show by induction on $M$ that this is sufficient to ensure the existence of nonnegative integer solutions to $({\rm E}_t)$. If $M=0$, we have the trivial solution $c_{ij}=0$ for all $i,j$. If $M>0$, there exists some $n+2\leq i_0\leq N$ such that $y_{i_0}(t)\geq 1$, and some $1\leq j_0\leq n+1$ such that $y_{j_0}(t)\geq 1$. If we replace $y_{i_0}(t)$ and $y_{j_0}(t)$ respectively by $y_{i_0}(t)-1$ and $y_{j_0}(t)-1$, we get a solution $(c_{ij})$ from the induction hypothesis, and it suffices to change $c_{i_0j_0}$ to $c_{i_0j_0}+1$ to obtain a solution for our original system.

Finally, for such a solution, none of the $c_{ij}$ can exceed $p-1$. Actually from Lemma \ref{suppmin}, when $\sigma(t)\leq n$, we have $x_j(t)=1$, and all $y_j(t)$, $1\leq j\leq n+1$, are at most $p-1$; we get the assertion from the first $n+1$ equations of $({\rm E}_t)$. If $\sigma(r)\geq n$, we conclude from the $N-n-1$ last equations in the same way. 
\end{proof}

As a consequence, the following result ensures the existence of minimal solutions. We add an additional hypothesis, but we will see farther that it is sufficient for our purposes.

\begin{lemma}
Assume that we have $\sigma(t_0)=n$ for some $0\leq t_0\leq m-1$. Then there exists $(\x(0),\ldots,\x(m-1))\in M_0\times\ldots\times M_{m-1}$ satisfying the systems $({\rm I}_t)$ for all $t$.
\end{lemma}

\begin{proof}
From $\sigma(t_0)=n$, we must choose $\x(t_0)=(1,\ldots,1)$ from Lemma \ref{suppmin}. 

We first construct $\x(t_0+1)$ satisfying $({\rm I}_{t_0})$, depending on $\sigma(t_0+1)$. 

If we have $\sigma(t_0+1)=n$, then we must choose $\x(t_0+1)=(1,\ldots,1)$; since all digits $b_{j}^{(m-t_0-1)}$ and $p-1-b_{i}^{(m-t_0-1)}$ are in $\{0,\ldots,p-1\}$, the system is satisfied.

Assume that $\sigma(t_0+1)<n$; then we have $x_j(t_0+1)=1$ for $1\leq j\leq n+1$, and reasoning as above, we see that the first $n+1$ inequalities in $({\rm I}_{t_0})$ are satisfied. It remains to show that there exists $x_i(t_0+1)$ for $n+2\leq i\leq N$ whose sum is $N-1-\sigma(t_0+1)$ and which satisfy the last $N-n-1$ inequalities. But it is sufficient to show that the sum of the right hand sides of these last inequalities is at least $N-1-\sigma(t_0+1)$. 

We evaluate this last sum: it is (we use Lemma \ref{shiftprop} (ii) here)
\begin{eqnarray*}
& & N-n-1+\sum_{i=n+2}^N b_{i}^{(m-t_0-1)}\\ & = & N-n-1+p\sigma(t_0)-\sigma(t_0+1)+p-1-\sum_{j=1}^{n+1} b_{j}^{(m-t_0-1)}\\
& = & N-1-\sigma(t_0+1)+(p-1)(n+1)-\sum_{j=1}^{n+1} b_{j}^{(m-t_0-1)}
\end{eqnarray*}
and we get the result since the terms of the lasr sum are digits, thus at most $p-1$.

Finally, if $\sigma(t_0+1)>n$, the last $N-n-1$ inequalities in $({\rm I}_{t_0})$ are satisfied, and it is sufficient to show that the sum of the right hand sides of the first $n+1$ ones is at least $\sigma(t_0+1)+1$. Here we have
\begin{eqnarray*}
& & p(n+1)-\sum_{j=1}^{n+1} b_{j}^{(m-t_0-1)}\\ & = & p(n+1)-\left(p(\sigma(t_0)+1)-(\sigma(t_0+1)+1)-\sum_{i=n+2}^N b_{i}^{(m-t_0-1)}\right)\\
& = & \sigma(t_0+1)+1+\sum_{i=n+2}^N b_{i}^{(m-t_0-1)}\\
\end{eqnarray*}

In any case, we can choose $\x(t_0+1)\in M_{t_0+1}$ such that the inequalities in $({\rm I}_{t_0})$ are satisfied.

Assume that we have constructed $\x(t_0+1)\in M_{t_0+1},\ldots,\x(k)\in M_{k}$ such that the inequalities $({\rm I}_{t_0}),\ldots,({\rm I}_{k-1})$ are satisfied. We look for $\x(k+1)\in M_{k+1}$ such that the system $({\rm I}_{k})$ is satisfied.

As above, we use a case by case reasoning, depending on the values of $\sigma(k),\sigma(k+1)$. 

If $\sigma(k)=n$, we reason as above.

If $\sigma(k+1)=n$, there is nothing to do since $\x(k+1)=(1,\ldots,1)$ satisfies $({\rm I}_{k})$: all the right-hand sides of the inequalities are positive integers.

We treat only one of the four remaining cases. The other ones are completely similar, and we leave their verification to the reader.

Assume that both $\sigma(k)$ and $\sigma(k+1)$ are less than $n$. The first $n+1$ inequalities of $({\rm I}_{k})$ are verified since the $x_j(k),x_j(k+1)$ equal $1$ in this case. In order to verify that an element of $M_{k+1}$ satisfies the last $N-n-1$, we just have to check that the sum of their last $N-n-1$ right-hand sides are at least $N-1-\sigma(k+1)$. From the description of the set $M_k$ and the equality from lemma \ref{shiftprop} (ii), we have
\begin{eqnarray*}
& & \sum_{i=n+2}^N \left(px_i(k)-(p-1- b_{i}^{(m-k-1)})\right)\\ 
& = & p(N-\sigma(k)-1)-(p-1)(N-n-1)+\sum_{i=n+2}^N b_{i}^{(m-k-1)}\\
& = & N-p(\sigma(k)+1)+(p-1)(n+1)+p(\sigma(k)+1)\\
& & -(\sigma(k+1)+1)-\sum_{j=1}^{n+1} b_{j}^{(m-k-1)}\\
& = & N-1-\sigma(k+1)+(p-1)(n+1)-\sum_{j=1}^{n+1} b_{j}^{(m-k-1)}\\
\end{eqnarray*}
and the desired inequality, since each of the $b_{j}^{(m-k-1)}$ is at most $p-1$.

We have shown the lemma inductively.
\end{proof}

Summing up, we get the following result.

\begin{proposition}
Assume that there exists some $0\leq t\leq m-1$ such that $\sigma(t)=n$. Then the set $M_n(\gamma,q)$ is non empty.
\end{proposition}

\subsection{Minimal solutions for varying fields}

In this section, we show that the minimal solutions modulo any power $q$ of $p$ such that $q\equiv 1\mod d$ are completely determined by those modulo $p^r$, under an hypothesis similar to the one from the preceding proposition.

\begin{proposition}
\label{minsol}
Assume that $\sigma(0)=n$; then for any $s\geq 1$, we have
\[
M_n(\gamma,p^{rs})=M_n(\gamma,p^r)+p^rM_n(\gamma,p^r)+\ldots+p^{r(s-1)}M_n(\gamma,p^r)
\]
\end{proposition}

\begin{proof}
We set $q:=p^{rs}$. First remark that for $d_k:=\frac{p^r-1}{d}a_k$, $1\leq k\leq N$, we have $b_k=\frac{q-1}{p^r-1}a_k=(1+p^r+\ldots+p^{r(s-1)})d_k$. In other words, the digits of the base $p^r$ expansion of $b_k$ are all equal to $d_k$, and the same is true for $q-1-b-k$, with the unique digit $p^r-1-d_k$. Moreover, the vector $\d_n$ constructed as $\b_n$, but from the $d_k$, satisfies $\b_n=(1+p^r+\ldots+p^{r(s-1)})\d_n$.

Choose $U^{(0)},\ldots,U^{(s-1)}\in M_n(\gamma,p^r)$ minimal solutions, and set $U:=\sum_{t=0}^{s-1}p^{rt}U^{(t)}$ (the combination is coordinate by coordinate). Then we have
\[
\f_n(U)+\b_n=\sum_{t=0}^{s-1}p^{rt}(\f_n(U^{(t)})+\d_n)=\sum_{t=0}^{s-1}p^{rt}(p^r-1)(1,\ldots,1)=(q-1)(1,\ldots,1)
\]
where the second equality comes from our hypothesis $\sigma(0)=n$ and Proposition \ref{suppmin}. Thus $U$ is a solution, and its $p$-weight is the sum of those of the $U^{(i)}$, that is $s$ times the $p$-weight of a minimal solution in $A_n(\gamma,p^r)$. But this is the weight of a minimal solution in $A_n(\gamma,q)$ from the description of the digits of the $b_i$ above and the $r$-periodicity of $\sigma$. We have proved one inclusion.

We come to the other one : let $U\in M_n(\gamma,q)$ denote a minimal solution, and write its base $p^r$ expansion $U=\sum_{t=0}^{s-1}p^{rt}U^{(t)}$ (this is just the base $p^r$ expansion of each coordinate). 

Since we assumed $\sigma(0)=n$, we must have $\varphi_U(0)=(1,\ldots,1)$ from Proposition \ref{suppmin}. Moreover, we have seen that $\sigma$ is $r$-periodic; thus we have $\varphi_U(tr)=(1,\ldots,1)$ for all $t\in \{0,\ldots,s-1\}$. 

We deduce, reasoning as at the beginning of the proof of Lemma \ref{red1}, that for any $t\in \{0,\ldots,s-1\}$, we have the equality
\[
\f_n(U^{(s-t-1)})+\d_n=p^r\varphi_U(tr)-\varphi_U((t+1)r)=(p^r-1)(1,\ldots,1)
\]
where we have used the description of the $p^r$-digits of $\b_n$. Thus each of the $U^{(t)}$ is in $A_n(\gamma,p^r)$. Finally, since the $p$-weight of $U$ is the sum of those of the $U^{(t)}$, we deduce from the minimality of $U$ that each of the $U^{(t)}$ is minimal. We are done.
\end{proof}

\section{The generic Newton polygon for $L$-functions}
\label{sec3}

In this section, we consider the family of polynomials $g(x)=\prod_{i=1}^N (x-\alpha_i)^{a_i}$ when the $\alpha_i$ vary in $\overline{\F}_p$, subject to the condition $\alpha_i\neq \alpha_j$ when $i\neq j$. If we fix $\chi$ a multiplicative character of order $d$ of a field $\F_q$ containing all $\alpha_i$ and the $d$-th roots of unity, we get a family of $L$-functions $L(g,\chi;T)$. It follows from Grothendieck's specialization theorem that there exists a polygon, the generic Newton polygon, such that for any $g$, the Newton polygon of $L(g,\chi;T)$ lies above it, and they are equal over a non empty Zariski open subset of the subvariety of the affine space $\A^N=\Spec \overline{\F}_p[\alpha_1,\ldots,\alpha_N]$ parametrising the family. 

We show that the generic Newton polygon for the above family of $L$-functions is the combinatorial polygon from Definition \ref{defpoly}, and for any vertex of this polygon, we determine for which $g$ it is a vertex of the Newton polygon of $L(g,\chi;T)$.

\subsection{Congruences for the coefficients, and the Hasse polynomials}

We fix $\alpha_1,\ldots,\alpha_N$ and a field $\F_q$ containing them and the $d$-th roots of unity 

We first use (\ref{exprSn}) and a more precise version of Stickelberger's theorem on Gauss sums to determine a congruence for the degree $n$ coefficient of the $L$-function, where $n$ is the abscissa of a vertex $(n,B(n))$ of the polygon $\Pi(\gamma,e)$.

If we replace the coefficients $c_k$ of the Ax polynomial $C$ by their expressions in terms of Gauss sums, and use the equality $S_n=(q-1)L_n$, the formula (\ref{exprSn}) becomes
\[
L_n=(q-1)^{n(n+2-N)}\left(\prod_{i=n+2}^N g(-b_i)\right)^{-1}\sum_{U\in A_n(\gamma,q)}\prod_{j=1}^{n+1}\prod_{i=n+2}^{N} g(-u_{ij})\omega^{u_{ij}}(\alpha_{ij}) 
\]
Now recall Stickelberger's congruence: for any $1\leq u\leq q-2$ with base $p$ expansion $u=\sum p^tu^{(t)}$, we have
\[
g(-i)\equiv \frac{\pi^{s_p(u)}}{u^{(0)}!\ldots u^{(m-1)}!} \mod \pi^{s_p(u)+1}
\]
where $\pi=\psi(1)-1$ is the generator of the maximal ideal of $\Z_p[\zeta_p,\zeta_{q-1}]$ chosen above.

In order to have more compact notations, we set $u!!:=u^{(0)}!\ldots u^{(m-1)}!$ and $0!!=-1$ for the formula to extend to the case $i=0$.

As we have already seen, the term in the preceding expression for $L_n$ corresponding to $U$ has $\pi$-adic valuation $s_p(U)$. This weight is at least $mB(n)+\sum s-p(b_i)$, and this is an equality if and only we have $U\in M_n(\gamma,q)$. 

Using the fact that the coefficient $L_n$ is in $\Z_p[\zeta_{q-1}]$ whose maximal ideal is generated by $p$, and that $\pi^{p-1}=-p$, we obtain the congruence

\[
L_n\equiv \pm\prod_{i=n+2}^N b_i!!\left(\sum_{U\in M_n(\gamma,q)}\prod_{j=1}^{n+1}\prod_{i=n+2}^{N} \frac{\omega^{u_{ij}}(\alpha_{ij})}{u_{ij}!!}\right)(-p)^{mB(n)} \mod p^{mB(n)+1}
\]

We focus on the sum between parentheses in the above formula. Actually the vertex of abscissa $n$ of the Newton polygon of the $L$-function and of the combinatorial polygon coincide if and only if the $q$-adic valuation of $L_n$ is $B(n)$ if and only if the reduction modulo $p$ in $\F_q$ of this sum is non zero.

We now determine that reduction; note that since the Teichmüller character is a section of reduction modulo $p$, we just have to remove $\omega$ to get a first expression. Then we replace the $\alpha_{ij}$ by their values; recall these are evaluations of Lagrange interpolation polynomials
\[
\alpha_{ij}=\prod_{k=1,k\neq j}^{n+1} (\alpha_i-\alpha_k) (\alpha_j-\alpha_k)^{-1}
\]
Using the congruences $\sum_i u_{ij}\equiv b_i\mod q-1$, $\sum_j u_{ij}\equiv -b_j\mod q-1$, we get 
\[
\prod_{j=1}^{n+1}\prod_{i=n+2}^{N}\alpha_{ij}^{u_{ij}}=\prod_{k=1}^{n+1}\prod_{l=1,~ l\neq k}^{N}(\alpha_l-\alpha_k)^{b_l}\prod_{j=1}^{n+1}\prod_{i=n+2}^{N}(\alpha_i-\alpha_j)^{-u_{ij}}.
\]
If we get rid of the (non zero) double product over $(k,l)$ that does not depend on $U$, we get the following

\begin{lemma}
\label{vertexcoincidence}
The Newton polygon of the $L$-function and of the combinatorial polygon coincide at the vertex with abscissa $n$ if, and only if we have
\[
\sum_{U\in M_n(\gamma,q)}\prod_{j=1}^{n+1}\prod_{i=n+2}^{N} \frac{(\alpha_{i}-\alpha_j)^{-u_{ij}}}{u_{ij}!!} \in \F_q^\times
\]
\end{lemma}

Note that since we assumed that $(n,B(n))$ is a vertex of the polygon $\Pi(\gamma,e)$, the map $\sigma$ takes the value $n$, an the index set of the sum is non-empty from the preceding section. Correctly normalized, this sum gives a first polynomial

\begin{definition}
\label{Hasseq}
Fix $\gamma,q$ as above, and let $\pi$ denote a vertex of the polygon $\Pi(\gamma,e)$ with abscissa $n$. 

The \emph{leading polynomial for the vertex $\pi$ and the field $\F_q$} is the polynomial
\[
\H_{\gamma,q}^{(\pi)}=\sum_{U\in M_n(\gamma,q)}\prod_{j=1}^{n+1}\prod_{i=n+2}^{N} \frac{(\alpha_{i}-\alpha_j)^{q-1-u_{ij}}}{u_{ij}!!} \in \F_p[\alpha_1,\ldots,\alpha_N]
\]
\end{definition}

We have seen that minimal solutions modulo $q$ can be expressed from minimal solutions modulo $p^r$. This gives a factorization for the above polynomials

\begin{proposition}
\label{normhasse}
Notations are as above. Assume moreover that $\sigma(t_0)=n$ for some $t_0\in\{0,\ldots,r-1\}$. 

Then for any $\alpha_1,\ldots,\alpha_N\in \F_q$, we have
\[
\H_{\gamma,q}^{(\pi)}(\alpha_1,\ldots,\alpha_N)=\Norm_{\F_q/\F_{p^r}}\left(\H_{\gamma,p^r}^{(\pi)}(\alpha_1,\ldots,\alpha_N)\right)
\]
\end{proposition}

\begin{proof}
We first assume that $t_0=0$, and we apply Proposition \ref{minsol}: in the definition of $\H_{\gamma,q}^{(\pi)}$, we replace the sum over $U\in M_n(\gamma,q)$ by $s$ sums over $U^{(k)}\in M_n(\gamma,p^r)$, where the $U^{(k)}$ are the base $p^r$ digits of $U$, as in the proof of Proposition \ref{minsol}.

We have $u_{ij}=\sum p^k u_{ij}^{(k)}$; we deduce the equalities $u_{ij}!!=\prod_k u_{ij}^{(k)}!!$, $q-1-u_{ij}=\sum p^k(p^r-1- u_{ij}^{(k)})$. As a consequence, the term of the sum corresponding to $U$ can be rewritten
\[
\prod_{j=1}^{n+1}\prod_{i=n+2}^{N}\prod_{k=0}^{s-1} \left( \frac{(\alpha_{i}-\alpha_j)^{p^r-1-u_{ij}^{(k)}}}{u_{ij}^{(k)}!!}\right)^{p^k}
\]
Now we can use distributivity to get the factorization
\[
\H_{\gamma,q}^{(\pi)}(\alpha_1,\ldots,\alpha_N)=\prod_{k=0}^{s-1}\left(\sum_{U^{(k})\in M_n(\gamma,p^r)} \prod_{j=1}^{n+1}\prod_{i=n+2}^{N} \frac{(\alpha_{i}-\alpha_j)^{p^r-1-u_{ij}^{(k)}}}{u_{ij}^{(k)}!!}\right)^{p^k}
\]
and this is the desired result.

We come to the general case. We consider $\gamma'=(d,N,\a')$, $\a'=(a_1',\ldots,a_N')$ where $a_k'$ is the remainder of $p^{t_0}a_k$ modulo $d$. For any $\alpha_1,\ldots,\alpha_N$ as above, an easy variable change shows that for all $k\geq 1$, we have the equality of character sums $s_k(g)=s_k(g')$ where $g'(x):=\prod_{i=1}^N (x-\alpha_i^{p^{-t_0}})^{a_i'}$. We deduce the equality $L(g,\chi;T)=L(g',\chi,T)$, and for the vertex $\pi$ of abscissa $n$ (the combinatorial polygons associated to $\gamma$ and $\gamma'$ are the same from Remark \ref{polygonorbit})
\[
\H_{\gamma,q}^{(\pi)}(\alpha_1,\ldots,\alpha_N)=\H_{\gamma',q}^{(\pi)}(\alpha_1^{p^{-t_0}},\ldots,\alpha_N^{p^{-t_0}})
\]

We have $\sigma'(0)=n$ for the new signature associated to $\gamma'$, and we deduce from the first part of the proof
\[
\H_{\gamma,q}^{(\pi)}(\alpha_1,\ldots,\alpha_N)=\Norm_{\F_q/\F_{p^r}}\left(\H_{\gamma',p^r}^{(\pi)}(\alpha_1^{p^{-t_0}},\ldots,\alpha_N^{p^{-t_0}})\right)
\]
Now the term in the norm is nothing but $\H_{\gamma,p^r}^{(\pi)}(\alpha_1,\ldots,\alpha_N)$ from above; this concludes the proof.
\end{proof}

We have obtained a polynomial that no longer depends of the field containing the $\alpha_k$, but only on the characteristic 
\begin{definition}
\label{Hasse}
Fix $\gamma,p,\pi$ as above, we define the \emph{Hasse polynomial for the vertex $\pi$ in characteristic $p$} as the polynomial
$\H_{\gamma,p^r}^{(\pi)}$, and the \emph{Hasse polynomial for the polygon $\Pi(\gamma,e)$ in characteristic $p$} as 
\[
\H_{\gamma,p}:=\prod_{\pi\in \Sigma(\gamma,e)} \H_{\gamma,p^r}^{(\pi)}\in \F_p[\alpha_1,\ldots,\alpha_N]
\]
\end{definition}

\subsection{The generic Newton polygon}

We are ready to prove our main result about Newton polygons of $L$-functions.

\begin{theorem}
\label{main2}
The polygon $\Pi(\gamma,e)$ is the generic Newton polygon for the family of $L$-functions 
\[
L(\prod_{i=1}^N (x-\alpha_i)^{a_i},\chi,T),~ \alpha_1,\ldots,\alpha_N\in \overline{\F}_p,~ \alpha_i\neq \alpha_j
\] 

It is attained if, and only if we have $\H_{\gamma,p}(\alpha_1,\ldots,\alpha_N)\neq 0$.
\end{theorem}

\begin{proof}
We know that $\Pi(\gamma,e)$ is a lower bound for the generic Newton polygon from Proposition \ref{lowerbound}.

Thus it is sufficient to show that it is attained, and we prove only the second assertion. From Lemma \ref{vertexcoincidence} and Proposition \ref{normhasse}, the Newton polygon of the $L$-function shares the vertex $\pi$ with $\Pi(\gamma,e)$ if, and only if we have $\H_{\gamma,p^r}^{(\pi)}(\alpha_1,\ldots,\alpha_N)\neq 0$.

Guaranteeing that the product over the vertices is non zero gives a necessary and sufficient condition for the coincidence of the two polygons.
\end{proof}

\bibliographystyle{amsalpha}
\bibliography{MuOrdinarity}

\providecommand{\bysame}{\leavevmode\hbox to3em{\hrulefill}\thinspace}
\providecommand{\MR}{\relax\ifhmode\unskip\space\fi MR }
\providecommand{\MRhref}[2]{%
  \href{http://www.ams.org/mathscinet-getitem?mr=#1}{#2}
}
\providecommand{\href}[2]{#2}
\begin{thebibliography}{LMPT22}

\bibitem[AS14]{adsp}
Alan Adolphson and Steven Sperber, \emph{Hasse invariants and mod {{\(p\)}}
  solutions of {{\(A\)}}-hypergeometric systems}, J. Number Theory \textbf{142}
  (2014), 183--210.

\bibitem[Ax64]{ax}
James Ax, \emph{Zeroes of polynomials over finite fields}, Am. J. Math.
  \textbf{86} (1964), 255--261.

\bibitem[BF07]{bf}
R{\'e}gis Blache and {\'E}ric F{\'e}rard, \emph{Newton stratification for
  polynomials: the open stratum}, J. Number Theory \textbf{123} (2007), no.~2,
  456--472.

\bibitem[Bla12]{bla}
R{\'e}gis Blache, \emph{Valuation of exponential sums and the generic first
  slope for {Artin}-{Schreier} curves}, J. Number Theory \textbf{132} (2012),
  no.~10, 2336--2352.

\bibitem[Bou01]{bouw}
Irene~I. Bouw, \emph{The {{\(p\)}}-rank of ramified covers of curves.}, Compos.
  Math. \textbf{126} (2001), no.~3, 295--322.

\bibitem[DH34]{daha}
Harold Davenport and Helmut Hasse, \emph{Die {Nullstellen} der
  {Kongruenzzetafunktionen} in gewissen zyklischen {F{\"a}llen}}, J. Reine
  Angew. Math. \textbf{172} (1934), 151--182.

\bibitem[Dol13]{doll}
John Dollarhide, \emph{On the {{\(L\)}}-function of multiplicative character
  sums}, Trans. Am. Math. Soc. \textbf{365} (2013), no.~3, 1637--1668.

\bibitem[Kat79]{kasl}
Nicholas~M. Katz, \emph{Slope filtration of $f$-crystals}, Journ\'ees de
  G\'eom\'etrie Alg\'ebrique de Rennes, Ast\'erisque, no.~63, Soci\'et\'e
  math\'ematique de France, 1979, pp.~113--163.

\bibitem[Kat04]{kag2}
\bysame, \emph{Notes on {{\(G_{2}\)}}, determinants, and equidistribution},
  Finite Fields Appl. \textbf{10} (2004), no.~2, 221--269.

\bibitem[LMPT19]{lmpt}
Wanlin Li, Elena Mantovan, Rachel Pries, and Yunqing Tang, \emph{Newton
  polygons arising from special families of cyclic covers of the projective
  line}, Res. Number Theory \textbf{5} (2019), no.~1, 31, Id/No 12.

\bibitem[LMPT22]{lmpt2}
\bysame, \emph{Newton polygon stratification of the {Torelli} locus in unitary
  {Shimura} varieties}, Int. Math. Res. Not. (2022), no.~9, 6464--6511.

\bibitem[LMS24]{limasi}
Yuxin Lin, Elena Mantovan, and Deepesh Singhal, \emph{Abelian covers of
  {{\(\mathbb{P}^1\)}} of {{\(p\)}}-ordinary {Ekedahl}-{Oort} type}, Int. Math.
  Res. Not. (2024), no.~23, 14369--14392.

\bibitem[LZ05]{lizh}
Hanfeng Li and Hui~June Zhu, \emph{Zeta functions of totally ramified
  {{\(p\)}}-covers of the projective line}, Rend. Semin. Mat. Univ. Padova
  \textbf{113} (2005), 203--225.

\bibitem[Moo04]{most}
Ben Moonen, \emph{Serre-{Tate} theory for moduli spaces of {PEL} type},
  vol.~37, 2004, pp.~223--269.

\bibitem[Moo10]{moss}
\bysame, \emph{Special subvarieties arising from families of cyclic covers of
  the projective line}, Doc. Math. \textbf{15} (2010), 793--819.

\bibitem[Phi27]{phig}
Antonine Phigareau, \emph{The p-powers dividing certain exponential sums},
  Finite Fields and Their Applications \textbf{117} (2027), 102876.

\bibitem[Pri25]{prie}
Rachel Pries, \emph{The {Torelli} locus and {Newton} polygons}, Preprint,
  {arXiv}:2509.00998.

\bibitem[RR96]{rari}
M.~Rapoport and M.~Richartz, \emph{On the classification and specialization of
  {{\(F\)}}-isocrystals with additional structure}, Compos. Math. \textbf{103}
  (1996), no.~2, 153--181.

\bibitem[Sti90]{stic}
L.~Stickelberger, \emph{Ueber eine {Verallgemeinerung} der {Kreistheilung}.},
  Math. Ann. \textbf{37} (1890), 321--367.

\bibitem[VW13]{viwe}
Eva Viehmann and Torsten Wedhorn, \emph{Ekedahl-{Oort} and {Newton} strata for
  {Shimura} varieties of {PEL} type}, Math. Ann. \textbf{356} (2013), no.~4,
  1493--1550.

\end{thebibliography}

\end{document}